\documentclass[reqno,12pt,letterpaper]{amsart}
\usepackage[proof]{sdmacros}
\usepackage{soul}

\newcommand{\calO}{{\mathscr O}}

\newcommand{\RR}{{\mathbb R}}

\def\Krein{Kre\u{\char16}n}

\usepackage{subcaption}

\title{The scattering phase: seen at last}
\author{Jeffrey Galkowski}
\email{j.galkowski@ucl.ac.uk}
\address{Department of Mathematics, University College London, WC1H 0AY, UK}
\author{Pierre Marchand}
\email{pierre.marchand@inria.fr}
\address{Unit\'e de Math\'ematiques Appliqu\'ees de ENSTA Paris,
 91762 Palaiseau Cedex}
\author{Jian Wang}
\email{wangjian@email.unc.edu}
\address{Department of Mathematics, University of North Carolina, Chapel Hill, NC 27599}
\author{Maciej Zworski}
\email{zworski@berkeley.edu}
\address{Department of Mathematics, University of California, Berkeley, CA 94720}
\begin{document}

\begin{abstract}
  The scattering phase, defined as $ \log \det S ( \lambda ) / 2\pi i $ where
  $ S ( \lambda ) $ is the (unitary) scattering matrix, is the
  analogue of the counting function for eigenvalues when dealing with exterior domains and
  is closely related to \Krein{}'s spectral shift function.
  We revisit classical results on asymptotics of the scattering phase and point out that
  it is never monotone in the case of strong trapping of waves. Perhaps more importantly,
  we provide the first numerical calculations of scattering phases for non-radial scatterers.
  They show that the asymptotic Weyl law is accurate even at low frequencies and reveal effects of trapping such as lack of monotonicity. This is achieved by using the recent
  high level multiphysics finite element software \href{https://freefem.org/}{FreeFEM}.
\end{abstract}

\maketitle

\section{Introduction}

The scattering phase and its close relative, the spectral shift function, have been
studied by mathematicians at least since the work of Birman and \Krein{} \cite{bikr}.
In the case of radial scattering, the scattering phase is the sum of phase shifts which are a central and
classical topic in quantum scattering -- see for instance \cite[\S 6.4]{sak}.

The scattering phase is defined using the scattering matrix, $ S ( \lambda ) $, which is a unitary operator
mapping incoming waves to outgoing waves -- see \S \ref{s:form} and Figure~\ref{f:scat}.
Because of its structure, the determinant of $S(\lambda)$ is well defined and we put
\begin{equation}
  \label{eq:scatp}
  \sigma ( \lambda ) := \frac{1}{2 \pi i } \log \det S ( \lambda )  \in \mathbb R , \ \ \ \sigma ( 0 ) = 0,
\end{equation}
where the last condition fixes the choice of $ \log $.

The scattering phase, $ \sigma ( \lambda ) $, is
appealing to mathematicians since it is a replacement for the counting
function of eigenvalues for scattering problems -- see \cite[\S 2.6, \S 3.9]{dizzy} and
references given there. More precisely, as established by Jensen--Kato \cite{jeka}
and Bardos--Guillot--Ralston \cite{bagura}, $ \sigma(\lambda ) $ satisfies
\begin{equation}
  \label{eq:BK} \tr ( f ( - \Delta_{\mathbb R^n \setminus \mathscr O} ) - f ( - \Delta ))  =
  \int_0^\infty f ( \lambda^2 ) \sigma' ( \lambda ) d \lambda , \ \ f \in \mathscr S ( \mathbb R).
\end{equation}
Here, as in the rest of this paper, we specialized to the case of Dirichlet Laplacian,
$ \Delta_{\mathbb R^n \setminus \mathscr O} $ on $ \mathbb R^n \setminus \mathscr O $, where $ \mathscr O \Subset
  \mathbb R^n $ is an open set with a piecewise smooth boundary and connected complement.
(Strictly speaking, $ f ( - \Delta_{\mathbb R^n \setminus \mathscr O} ) $ and $ f ( - \Delta ) $
are defined on $ L^2 ( \mathbb R^n
  \setminus \mathscr O ) $ and $ L^2 ( \mathbb R^n ) $, respectively, using the spectral theorem,
but we  consider the former space as subspace of $ L ^2 ( \mathbb R^n ) $ using extension by $ 0 $.)

%\begin{center}
%\begin{figure}
%\includegraphics[width=7cm]{figures/scattering_phase/circle_adaptive_pml_lu_P2_pml_petsc.pdf}  \hspace{0.5cm} \includegraphics[width=7cm]{figures/geometries/circle.pdf} 
%\caption{\label{f:circle}
%Scattering phase and the corresponding geometry: previously know case of the disk}
%\end{figure}
%\end{center}

It could then be considered somewhat surprising that, to our knowledge, $ \sigma ( \lambda ) $
has only been exhibited for radial scatterers.  That is, there has never been any
form of an actual assignment, via a numerical approximation, of $ \lambda \mapsto \sigma ( \lambda ) $.
At the time when asymptotic formulae for $ \sigma ( \lambda ) $ were mathematically
investigated (see \S \ref{s:weyl}) it is safe to say that such numerical computation were
out of reach. Here we benefit from major advances in computational power and, in particular,
from the recent high level multiphysics finite element software \href{https://freefem.org/}{FreeFEM} --
see \S \ref{s:num}.

The numerical results for a variety of two dimensional scatterers $ \mathscr O $ are shown in our figures.
The main conclusions are:
\begin{itemize}
  \item The Weyl asymptotics for $ \sigma ( \lambda ) $ given in \eqref{eq:Weyl2} provide an
        accurate approximation starting at $ 0 $ energy; this accuracy is particularly striking in the
        case of non-trapping geometries -- see Figure \ref{f:nontrap}. They also appear remarkably accurate in trapping geometries.

  \item Strong trapping immediately causes lack of monotonicity of $ \sigma ( \lambda ) $ which
        in accordance with \eqref{eq:BW} is related to the presence of resonances near the
        real axis (as reviewed in \S \ref{s:BW}) -- see top Figure \ref{f:trap}.

  \item Mild trapping, illustrated in the two bottom Figures \ref{f:trap}, does not seem to
        destroy monotonicity but there is a visible effect from scattering resonances at least for low
        frequencies.

  \item For star shaped obstacles the scattering phase is monotone \cite{Ral}. This monotonicity
        is not known for non-trapping obstacles even though \cite{pepo} provided full asymptotic
        expansion for $ \sigma ( \lambda ) $; numerical examples suggest that $\sigma(\lambda)$ may always
        be monotone for non-trapping obstacles -- see Figure \ref{f:nontrap}.
        More experimentation would, however, be required for a firm conjecture.

\end{itemize}

\subsection{Weyl law for $ \sigma ( \lambda ) $}
\label{s:weyl}

Possibly the most striking result about the counting function for the eigenvalues of
the Dirichlet Laplacian, $ \Delta_{\mathscr O },  $ on a \emph{bounded} domain $ \mathscr O  \subset \mathbb R^n $ is the Weyl law: with
\[ N ( \lambda ) :=  | \Spec ( - \Delta_{\mathcal O }) \cap [0, \lambda^2 ] | , \]
\begin{equation}
  \label{eq:Weyl} \begin{split}  N ( \lambda ) &  =
    \frac{ \omega_n \vol( \mathscr O ) } { ( 2 \pi)^{n}}\lambda^n -
    \frac{  \omega_{n-1}  \vol ( \partial \mathscr O )} {4 ( 2 \pi)^{n-1}} \lambda^{n-1}
    + o ( \lambda^{n-1} ) , \end{split} \end{equation}
where $ \omega_n := \vol ( B_{\mathbb R^n } ( 0 , 1 ) ) $.
It was conjectured by Weyl in 1913 and established by Ivrii in 1980 (see \cite{sava} and \cite{ivre} for the history of this problem) under the assumptions that $ \partial \mathscr O $ is smooth
and the set of periodic orbits has measure zero (a generically valid fact expected to be true
for all $ \mathscr O $ with smooth boundaries).

The trace formula \eqref{eq:BK} shows that $ \sigma ( \lambda ) $ is
the exact analogue of $ N ( \lambda ) $ since
$ \tr f ( \Delta_{\mathscr O } ) = \int_0^\infty f ( \lambda^2 ) N' ( \lambda )  d \lambda $.
It is then natural to ask if \eqref{eq:Weyl} holds for $ \sigma ( \lambda ) $, with the understanding
that, in agreement with \eqref{eq:BK} we now consider renormalized volume of $ \mathbb R^n \setminus
  \mathscr O $. Hence the natural analogue of \eqref{eq:Weyl} is given by
\begin{equation}
  \label{eq:WeylS}
  \sigma ( \lambda )   =
  - \frac{ \omega_n \vol( \mathscr O ) } { ( 2 \pi)^{n}}\lambda^n -
  \frac{  \omega_{n-1}  \vol ( \partial \mathscr O )} {4 ( 2 \pi)^{n-1}} \lambda^{n-1}
  + o ( \lambda^{n-1} ) .  \end{equation}
The difficulty in obtaining \eqref{eq:WeylS}
stems from the fact that
classical Tauberian theorems used for \eqref{eq:Weyl} use monotonicity of $ N ( \lambda ) $.
As we will see in \S \ref{s:BW}, $ \sigma ( \lambda ) $ is \emph{not}, in general,
monotone.

However, for star-shaped obstacles $ \sigma' ( \lambda ) \leq 0 $ was established
by Helton--Ralston \cite{Ral} (see also \cite{kat}). This monotonicity allowed Jensen--Kato \cite{jeka} to obtain
the leading term in \eqref{eq:WeylS} in that case (the convex case was treated by
Buslaev \cite{bus}). For convex obstacles Majda--Ralston \cite{mara}
improved on \cite{jeka} by obtaining a three term asymptotic expansion of $ \sigma ( \lambda ) $.
Using advances in propagation of singularities for
obstacle problems (see \cite[Chapter 24]{H3} and references given there) Petkov--Popov \cite{pepo}
obtained a \emph{full} asymptotic expansion of $ \sigma ( \lambda )$ as $ \lambda \to \infty $.

The first proof of \eqref{eq:WeylS} for all obstacles (for which the conditions after \eqref{eq:Weyl} hold)
was given by Melrose \cite{Mel88} using his trace formula for scattering poles (see \cite[\S 3.10, \S
  3.13]{dizzy}). Since that formula holds only in odd dimension the same restriction was imposed.
This restriction was lifted using different methods by Robert \cite{rob}. (A proof in all dimensions following
Melrose's idea can be given using \cite{pz1}.)
In this historical account we only discussed the Dirichlet obstacle case. For more general perturbations
see, for instance, \cite{Ch0}.

Specialized to two dimensions, \eqref{eq:WeylS} becomes
\begin{equation}
  \label{eq:Weyl2}
  \sigma ( \lambda ) = - \frac{ | \mathscr O |}{4 \pi} \lambda^2 - \frac{ | \partial \mathscr O | }{ 4 \pi}
  \lambda + o ( \lambda ) .
\end{equation}
In the non-trapping case, in addition to further terms in~\eqref{eq:Weyl2}, there is an asymptotic formula for $ \sigma' ( \lambda ) $ \cite{pepo}. When a non-trapping $ \mathscr O $ has corners (i.e. has piecewise smooth, Lipschitz boundary) the following formula
is suggested by heat expansions for interior problems which can be found in~\cite{Ch:83,MaRo}:
\begin{equation}
  \label{eq:Weyl3}
  \sigma ( \lambda ) = - \frac{ | \mathscr O |}{4 \pi} \lambda^2 - \frac{ | \partial \mathscr O | }{ 4 \pi}
  \lambda +  \frac 1{ 24 } \sum_j \left( \frac {\theta_j } \pi - \frac \pi {\theta_j } \right) - \frac1{ 24 \pi}
  \int_{\partial \mathscr O } H ds + o ( 1 ),
\end{equation}
where $ \theta_j $ are the angles at the corners (measured from outside) and $ H $ is
the curvature (with the convention that $ H > 0$ for circles; we note that
if there are no corners and connected $ \mathscr O $, $ \int_{\partial \mathscr O } H ds = 2 \pi $). However, to our knowledge only the first asymptotic term of~\eqref{eq:Weyl3} is known rigorously in this case.

In the figures illustrating numerical results both asymptotic formulas are plotted against
the computed scattering phase and its derivative.  It is interesting to note that for most frequencies
$ \sigma' ( \lambda ) $ seems to agree with the asymptotic formula even in trapping cases.
This is similar to phenomena proved in the recent work of Lafontaine--Spence--Wunsch \cite{lasw}
and perhaps could be rigorously established by similar methods.

\begin{center}
  \begin{figure}
    \includegraphics[width=9cm]{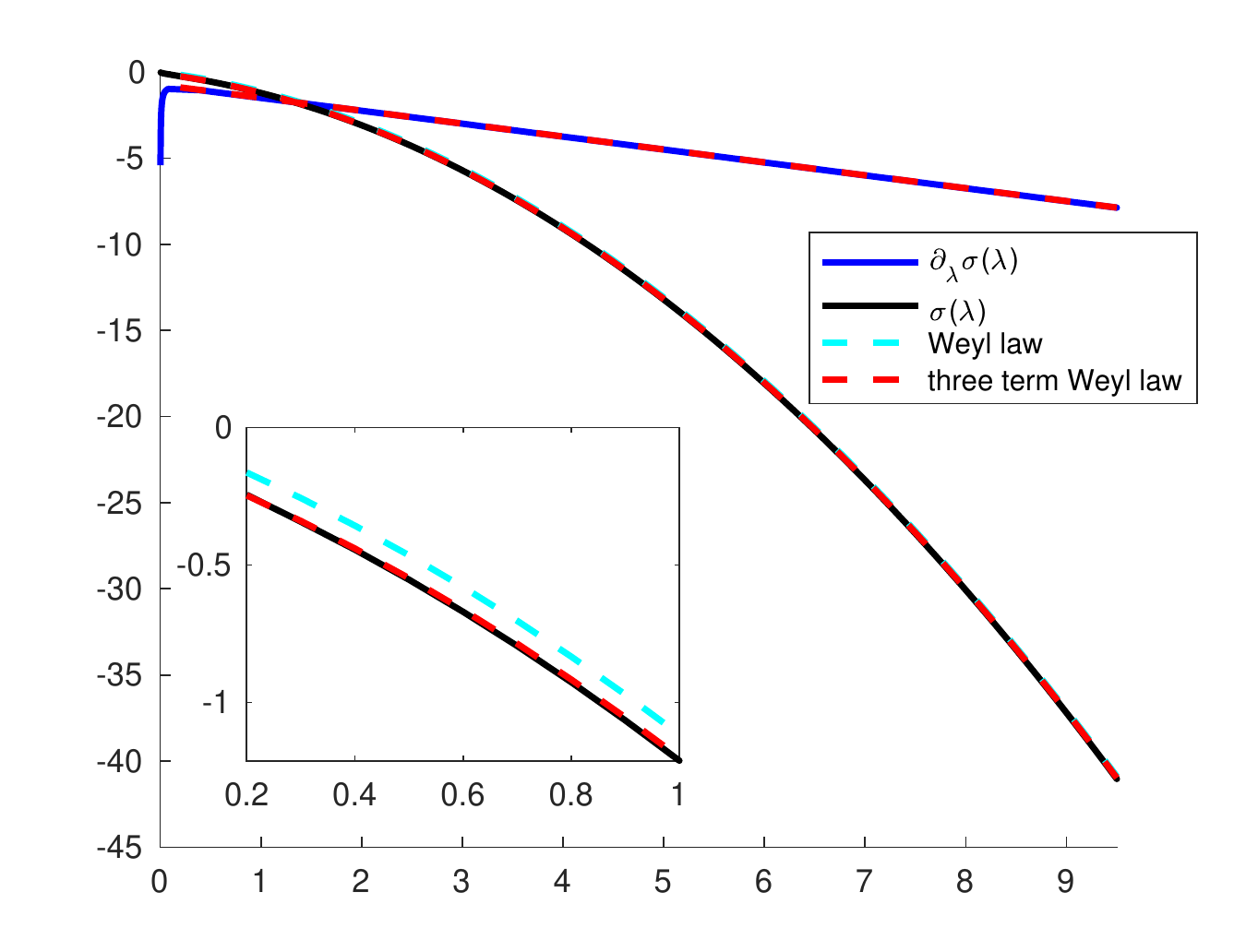}  \hspace{0.5cm} \includegraphics[width=5.5cm]{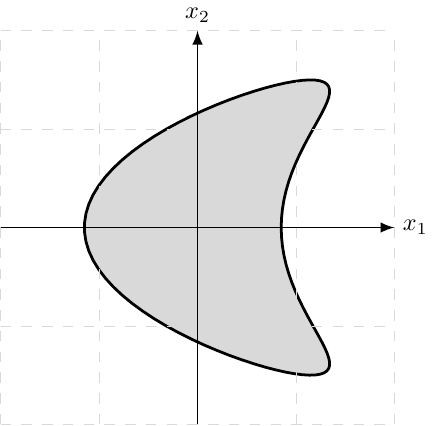}
    \includegraphics[width=9cm]{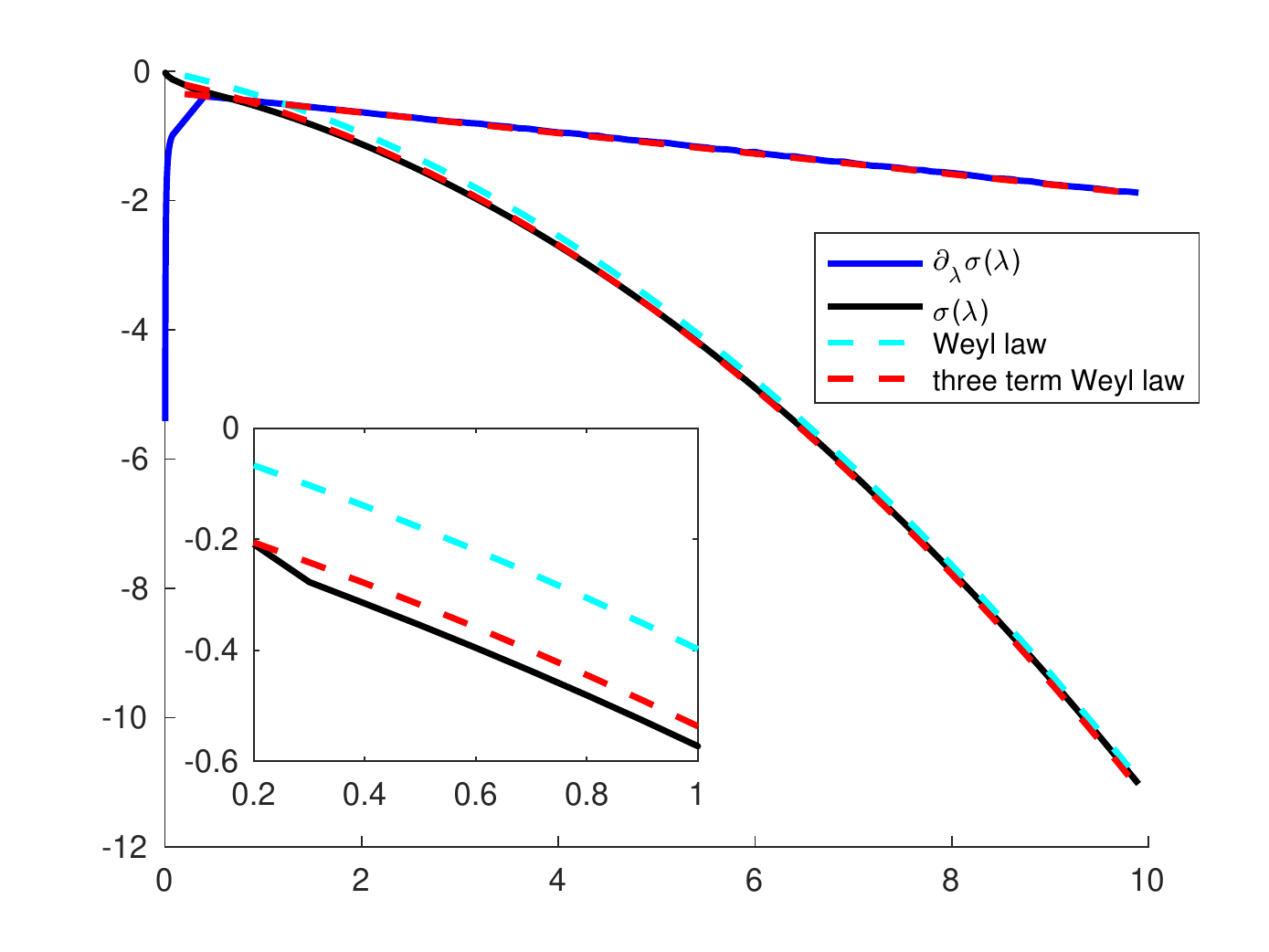}  \hspace{0.5cm} \includegraphics[width=5.5cm]{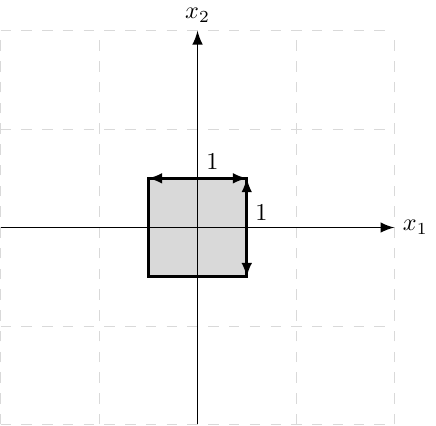}
    \includegraphics[width=9cm]{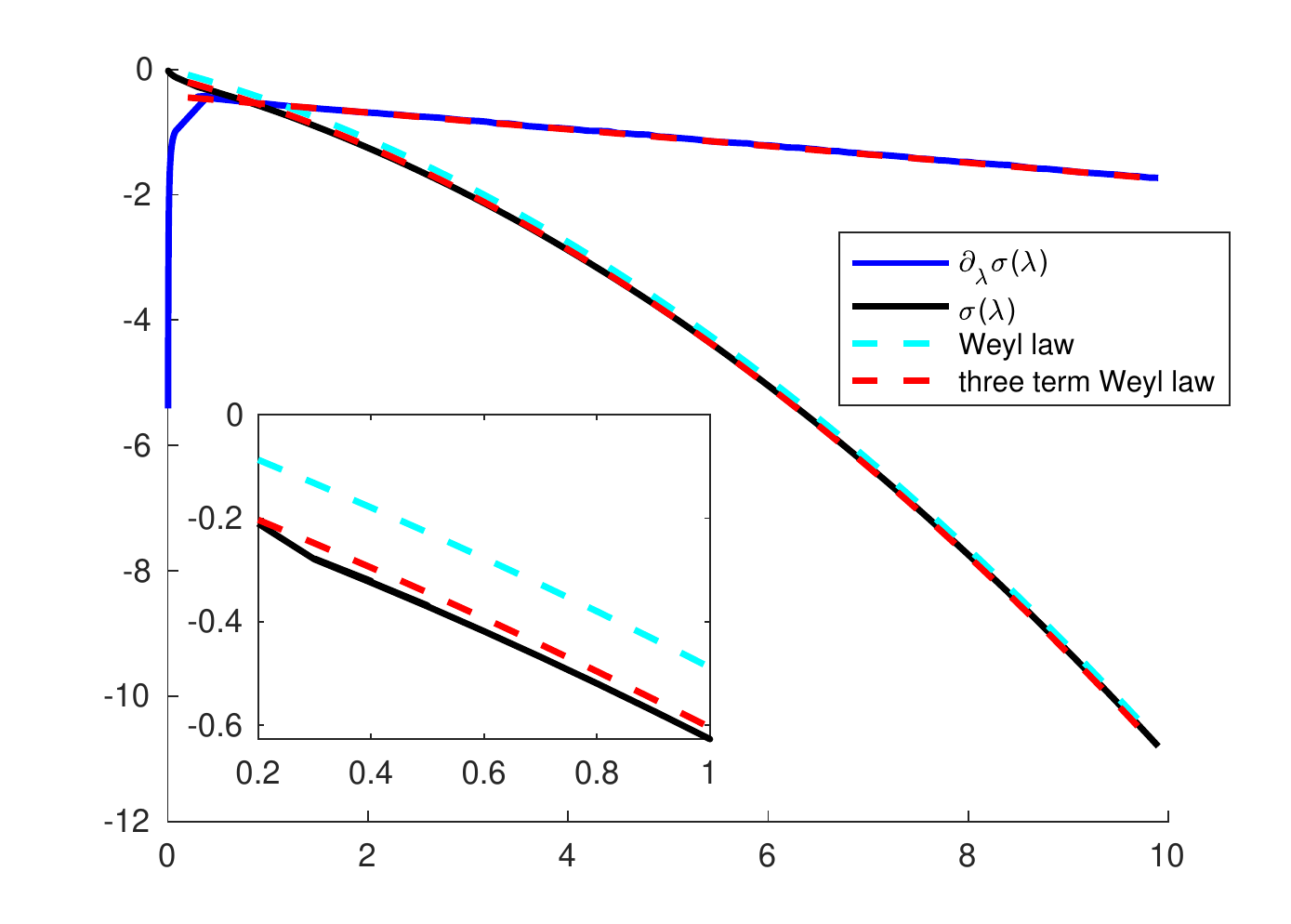}  \hspace{0.5cm} \includegraphics[width=5.5cm]{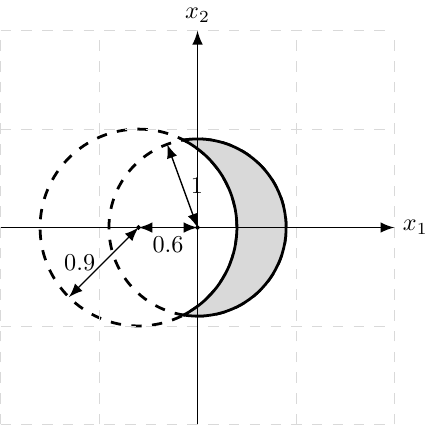}
    %\int H=2*2.0461-1/.9*2*1.4127
    %\includegraphics[width=7cm]{figures/scattering_phase/triangle_adaptive_pml_lu_P2_pml_petsc.pdf}  \hspace{0.5cm} \includegraphics[width=7cm]{figures/geometries/triangle.pdf} 
    \caption{\label{f:nontrap}
      Scattering phase and the corresponding geometry: from top to bottom,
      a star-shaped obstacle, a star-shaped obstacle with corners, a non-trapping non-starshaped obstacle. We also indicate the comparisons with the Weyl law \eqref{eq:Weyl2} and
      the (conjectural) three term Weyl for obstacles with corners \eqref{eq:Weyl3}.}
  \end{figure}
\end{center}

\begin{center}
  \begin{figure}
 \includegraphics[width=8.6cm]{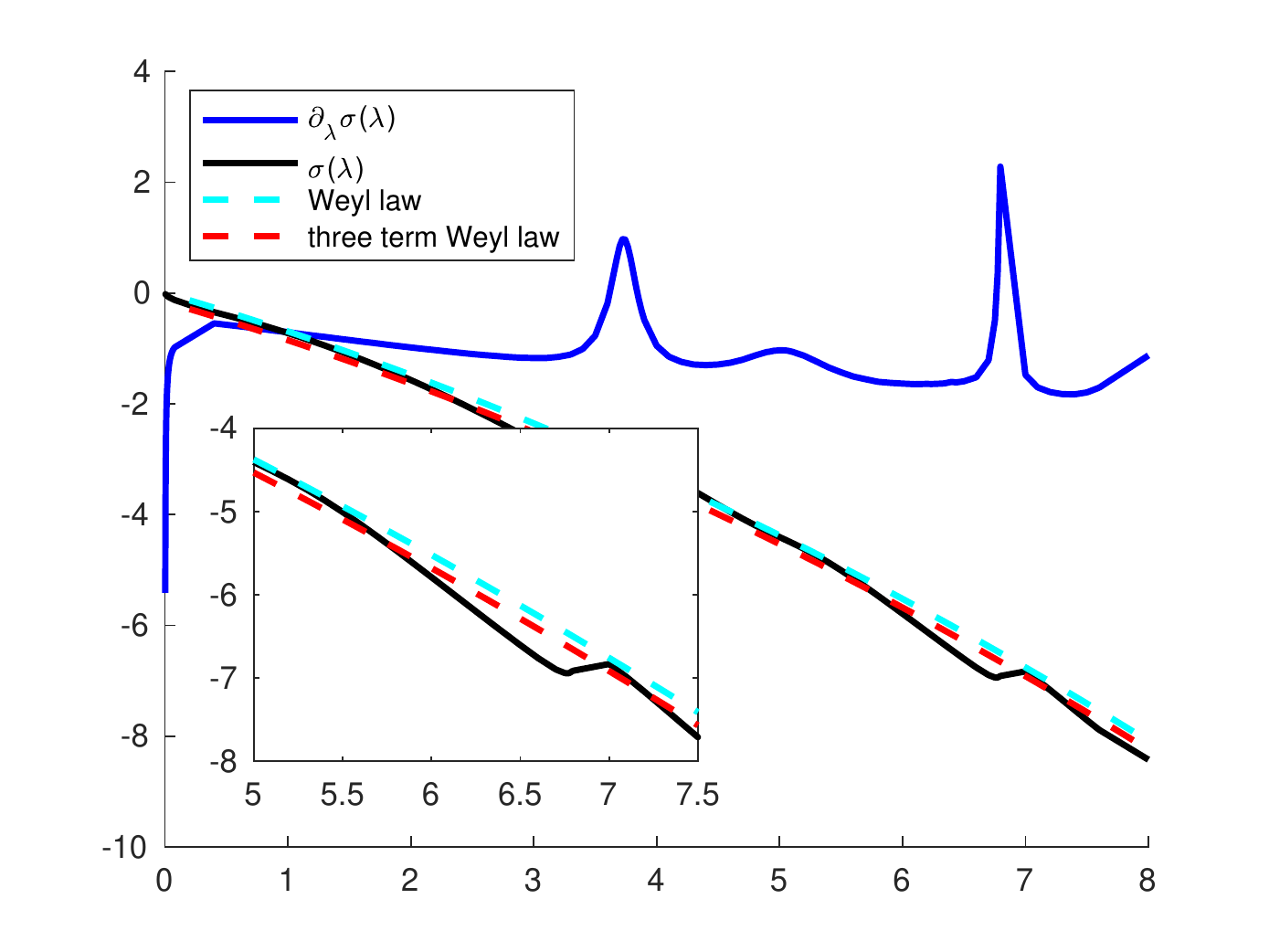}  \hspace{0.5cm} \includegraphics[width=5.25cm]{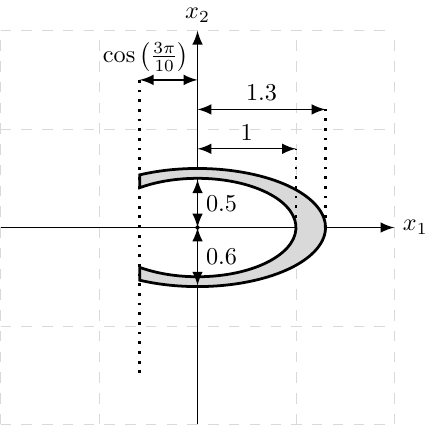}
    %\int H= 3.6013-3.8385
    \includegraphics[width=8.75cm]{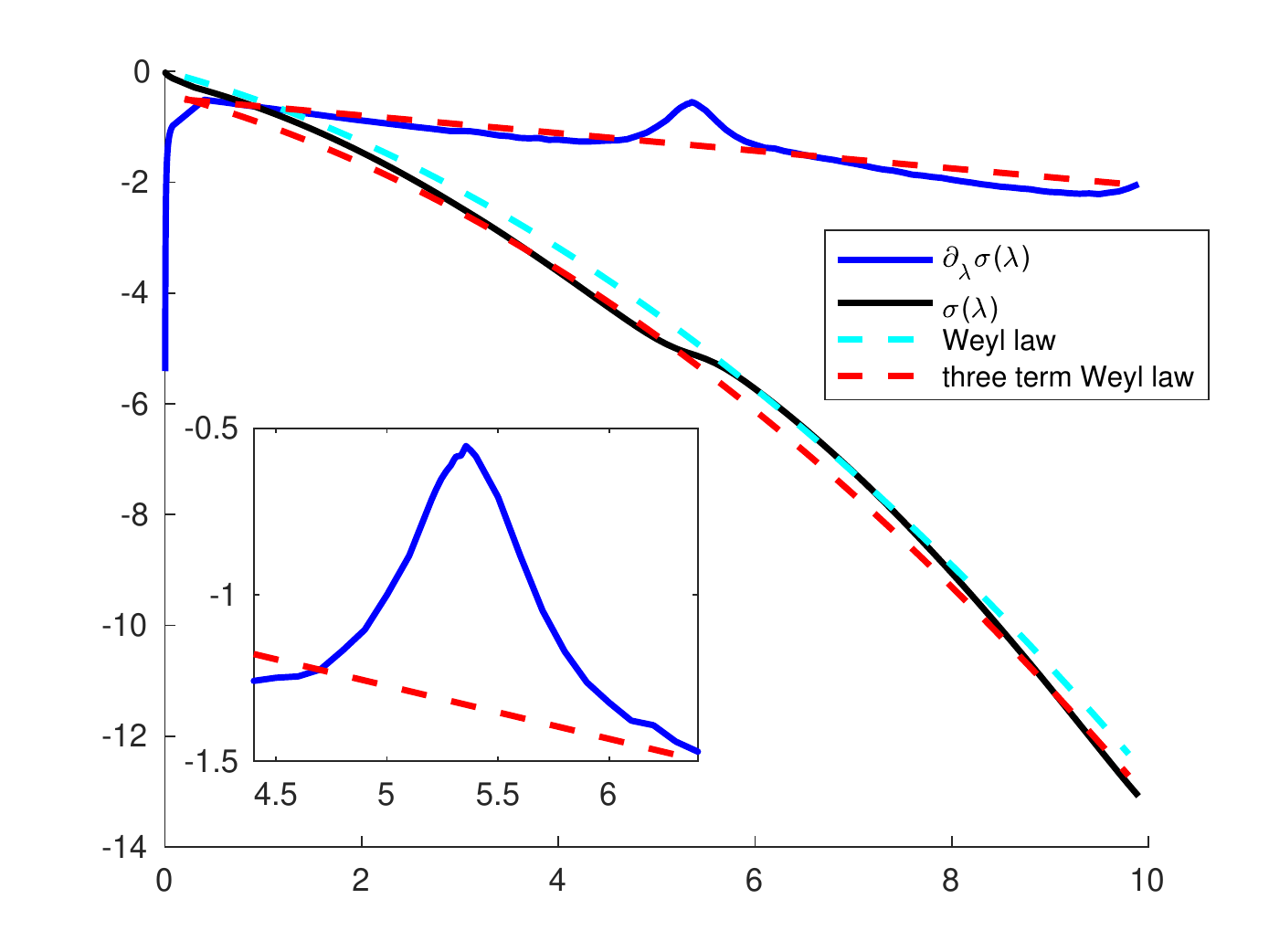}  \hspace{0.5cm} \includegraphics[width=5.25cm]{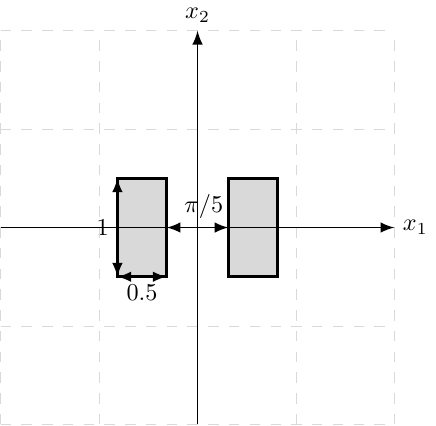}
    \includegraphics[width=8.75cm]{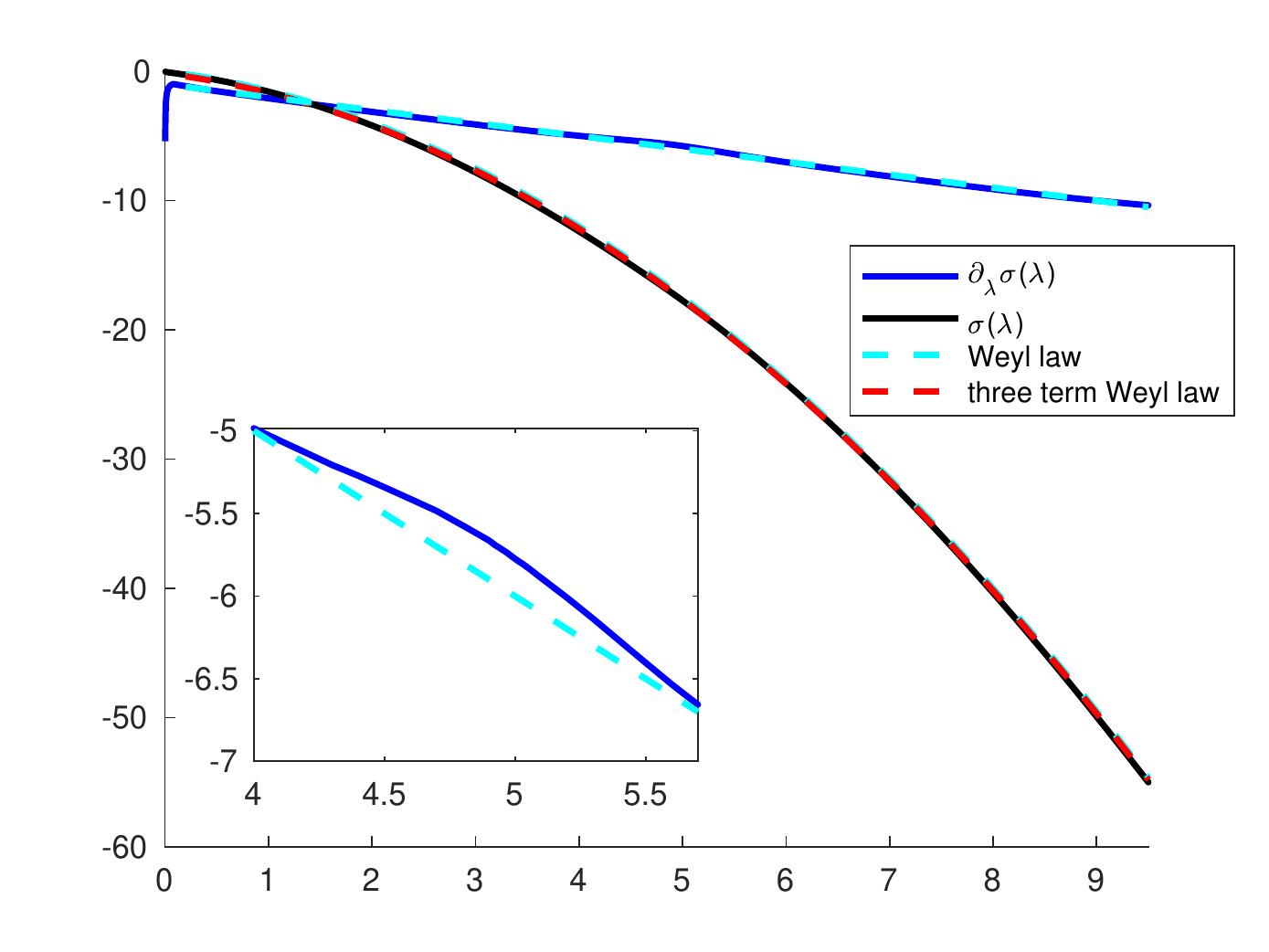}  \hspace{0.4cm} \includegraphics[width=5.25cm]{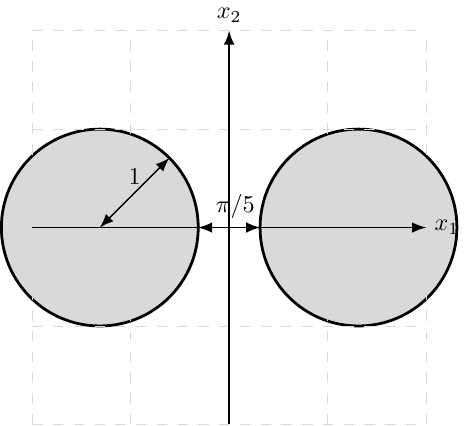}
    \caption{\label{f:trap} Scattering phase and the corresponding geometry:
      from top to bottom: strong trapping in a cavity, parabolic trapping from bouncing ball orbits, hyperbolic trapping in the form one closed orbit. In the case of strong trapping,
      we see numerical manifestations of \eqref{eq:BW}.
      For the two rectangles, we expect resonances
      with $ |\Im \lambda_j| \sim 1/|\lambda_j |$ so that \eqref{eq:BW} is inconclusive.
      In the case of two or more discs,  the resonances
      satisfy $ |\Im \lambda_j| > c $ (see \cite{vac} and references given there) and, as a result, at high
      energies their effect is weak.}
  \end{figure}

\end{center}

\subsection{Breit--Wigner approximation at high energies}
\label{s:BW}
Scattering resonances, which replace discrete spectral data for problems on unbounded domains,
can be defined (in obstacle scattering) as poles of the meromorphic continuation of $ S ( \lambda ) $
-- see \cite[\S 4.4]{dizzy}.
Since
$ S( \lambda ) $, $ \lambda > 0 $
captures observable phenomena, it is interesting to see how those (complex) poles manifest
themselves in its behaviour. The Breit--Wigner formula (see \cite[\S 2.2]{dizzy}) is one such way.
In high energy obstacle scattering it was proved by Petkov--Zworski \cite{pz1} and takes
the following form:
\begin{equation}
  \label{eq:BW}
  \sigma' ( \lambda )  = \sum_{ | \lambda_j - \lambda | < 1 }
  \frac 1 \pi \frac{| \Im \lambda_j | }{ |\lambda - \lambda_j |^2 } + \mathcal O (\lambda^{n-1} ) ,
\end{equation}
where $ \lambda_j $'s are the scattering resonances, that is the poles of $ S ( \lambda ) $.
From the point of view of the scattering asymptotics \eqref{eq:WeylS}
we note that the sign of the Breit--Wigner
terms (the sum of Lorentzians on the right in \eqref{eq:BW}) is opposite of the overall
trend. In particular, if there exist $ \lambda_j $'s with
$ |\Im \lambda_j| \ll (\Re \lambda_j)^{1-n} $, then $ \sigma' ( \lambda ) > 0 $ for
$ \lambda $ near $ \Re \lambda_j $. Strong trapping, such as that shown in Figure~\ref{f:trap} (top figure),
is known to produce resonances with $ \Im \lambda_j = \mathcal O ( |\lambda_j|^{-\infty } ) $
-- see \cite{StefDuke}, \cite{TaZ}. Consequently, whenever such strong trapping occurs
the scattering phase is \emph{not} monotone.

The strong and parabolic trapping examples in Figures \ref{f:trap} (top two figures) show the
presence of Lorentzians in $ \sigma' $ already at low energies. In the very weak
trapping illustrated in in the bottom Figure \ref{f:trap} there is some evidence of a low energy
resonance but the effect seems minimal.

\subsection{Low energy asymptotics}
The numerical methods used to compute $ \sigma' ( \lambda ) $  are
not effective at very low energies -- see \S \ref{s:num}. To obtain $ \sigma ( \lambda ) $
by integration we used low energy asymptotic formulae for $ \sigma' ( \lambda ) $.
There has been recent progress on this subject and it is natural to review it here.

The first result we are aware of was obtained by Hassell--Zelditch \cite{haze} (using
monotonicity of $ \sigma ( \lambda ) $ as a function of the obstacle \cite{Ral}) and
stated that $ \sigma ( \lambda ) \sim \frac{1}{ 2} \log \lambda $. That was a by-product of their work
on planar obstacles with the same scattering phase (an analogue of the isospectral problem).
This result was successively improved by McGillivray \cite{mcg}, Strohmaier--Waters \cite{StW}
and Christiansen--Datchev \cite{CD} and a more precise asymptotic formula is given by
\begin{equation}
  \label{eq:CD}
  \sigma' ( \lambda ) \sim - \frac{2}{\lambda} \frac{1}{ ( -2 \log 2 \lambda + C ( \mathscr O ) + 2 \gamma  )^2 + \pi^2 }
\end{equation}
with $ C( \mathscr O ) $ the logarithmic capacity of $ \mathscr O$ (see below) and
$ \gamma $ the Euler constant. One way to define $ C ( \mathscr O ) $ is to consider the
Green function of $ \mathscr O $:
\[ - \Delta G (x) = 0 , \ \ x \in \mathbb R^2 \setminus \mathscr O , \ \ \
  G ( x ) = 0 , \ \ x \in \partial \mathscr O , \ \ \ G ( x ) \sim \log |x| , \ \ |x| \to \infty , \]
Then
\[  G ( x ) = \log | x| -  C ( \mathscr O ) + o ( 1 ) , \ \ |x| \to \infty.  \]
We only used the leading term to enhance the numerics.

\begin{center}
  \begin{figure}
    \includegraphics[width=10cm]{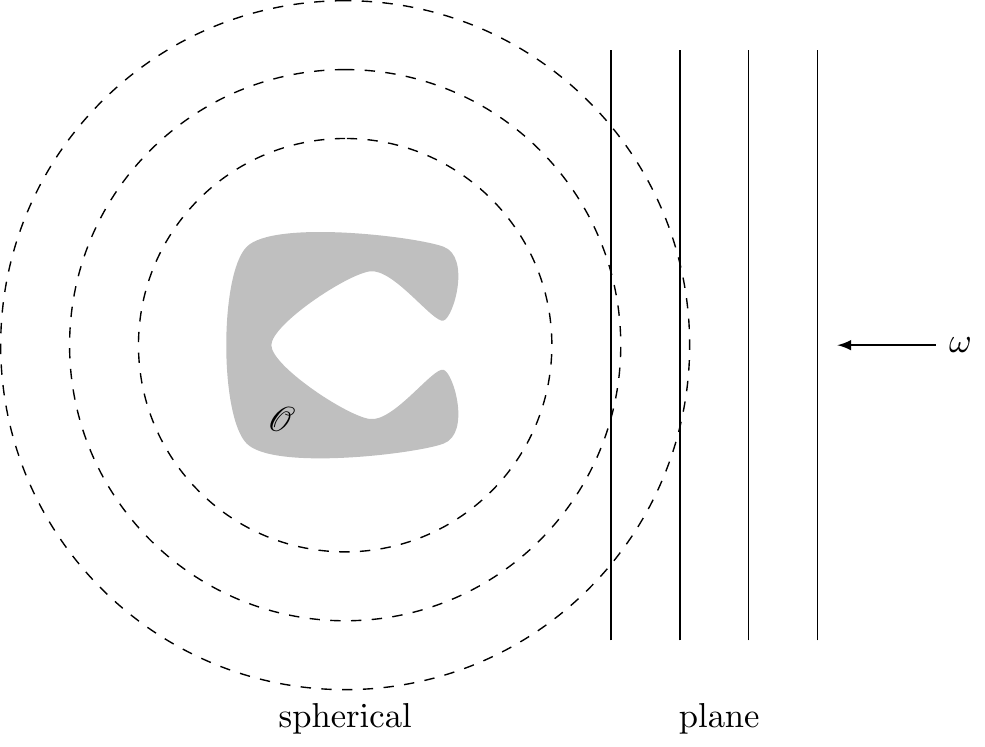}
    \caption{\label{f:scat} The waves used to define the scattering matrix}
  \end{figure}
\end{center}

\medskip
\noindent \textbf{Acknowledgements.} The authors are grateful to Euan Spence for helpful conversations at the beginning of the project.  JG was partially supported by EPSRC Early Career Fellowship EP/V001760/1 and Standard Grant EP/V051636/1, PM was partially supported by EPSRC grant EP/R005591/1, and MZ was partially supported by NSF grant DMS-1952939.

\section{A formula for the derivative of the scattering phase}
\label{s:form}

In order to compute $\sigma(\lambda)$ we recall a definition
of the scattering matrix  in dimension $n=2$ -- for
motivation and a detailed presentation see \cite[\S 3.7, \S 4.4]{dizzy}.

We start with perturbed plane waves -- see \eqref{e:defE} below.
For that we let $\omega\in \mathbb{S}^{1}$, $\lambda\in \mathbb{R}$ and define $u(\lambda,\cdot, \omega)\in C^\infty(\mathbb{R}^2)$ as the unique \emph{outgoing} solution to
\begin{equation}
  \label{eq:Helm}
  (-\Delta-\lambda^2)u=0\,\,\text{ in }\mathbb{R}^2\setminus \mathscr{O},\qquad u|_{\partial\mathscr{O}}= -e^{i\lambda \langle x,\omega\rangle}|_{\partial \mathscr{O}}.
\end{equation}
(We note that, to streamline notation, the convention is slightly different than in \cite{dizzy}.)
Here, by \emph{outgoing}, we mean that there is $b(\lambda,\cdot, \omega)\in C^\infty(\mathbb{S}^{1})$ such that
\begin{equation}
  \label{e:defU}
  u(\lambda, x,\omega)= e^{-\frac{\pi i}{4}}\sqrt{{2\pi}/({\lambda |x|})}e^{i\lambda |x|}b(\lambda, {x}/{|x|},\omega)+O(|x|^{-3/2}).
\end{equation}
We then define
\begin{equation}
  \label{e:defE}
  e(\lambda, x,\omega):= e^{i\lambda\langle x,\omega\rangle}+u(\lambda,x,\omega).
\end{equation}

The scattering matrix, $S(\lambda):L^2(\mathbb{S}^1)\to L^2(\mathbb{S}^1)$, is then given by $S(\lambda):=I+A(\lambda)$, where $A(\lambda)$ is an integral operator
defined as
\begin{equation}
  \label{eq:ala}
  A ( \lambda ) f ( \theta ) := \int_{\mathbb S^1}
  A(\lambda, \theta,\omega) f ( \omega ) d \omega,  \ \
  A(\lambda, \theta,\omega) :=b(\lambda,\theta,\omega).
\end{equation}
The scattering matrix $ S ( \lambda ) $ is unitary and extends meromorphically to
the Riemann surface of $ \log \lambda $.

It will be useful when computing the scattering phase to rewrite the integral kernel
$ A ( \lambda, \theta, \omega ) $ as an integral over $\partial\mathscr{O}$:
\begin{lemm}
  Let $\nu$ denote unit normal to $\partial\mathscr{O}$ pointing out of $\mathscr{O}$. Then, in the
  notation of \eqref{e:defE},  we have (with $ ds( x ) $ the line measure on $ \partial \mathscr O $ or
  $\partial B ( 0 , r) $)
  \begin{equation}
    \label{eq:Ala}
    A ( \lambda, \theta, \omega ) = \frac{1}{4\pi i}\int_{\partial\mathscr O} e^{-i\lambda\langle x,\theta\rangle}\partial_\nu e(\lambda,x,\omega)ds(x) .
  \end{equation}
\end{lemm}
\begin{proof}
  Green's formula shows that, with $ e (x)   := e(\lambda,x,\omega) $ and $\mathscr O \subset B( 0 , R )$
  \begin{equation}
    \label{e:boundaryPair}
    \begin{aligned}
      0 & =
      \int_{B(0,R)\setminus\mathscr{O}}\left(
      [(-\Delta-\lambda^2)e(x)](e^{-i\lambda\langle x,\theta\rangle}) -
      e(x)[ (-\Delta-\lambda^2)e^{-i\lambda\langle x,\theta\rangle})] \right) dx
      \\
        & =\int_{\partial\mathscr{O}}e^{-i\lambda\langle x,\theta\rangle}\partial_\nu e ( x ) ds(x)
      -
      \int_{\partial B(0,R)} \left( \partial_r e ( x ) e^{-i\lambda\langle x,\theta\rangle}-e( x) \partial_r [e^{-i\lambda\langle x,\theta\rangle}]\right) d s(x) .
    \end{aligned}
  \end{equation}
  To compute the last term in~\eqref{e:boundaryPair}, we use the formulae~\eqref{e:defU} and~\eqref{e:defE} together with the stationary phase method
  (see \cite[Theorem  {3.38}]{dizzy}): for $ a \in C^\infty ( \mathbb S^1 ) $,
  \begin{equation}
    \label{eq:statm}
    \int _{\partial B(0,R)}a({x}/{|x|})e^{-i\lambda \langle x,\theta\rangle}ds(x)=\sqrt{2\pi R/ \lambda }(e^{{-\frac{i\pi}{4}}}a(-\theta)e^{i\lambda R}+e^{\frac{i\pi}{4}}a(\theta)e^{-i\lambda R}) +\mathcal O(R^{-\frac12}).
  \end{equation}
  By applying \eqref{eq:statm} when $\theta\neq \omega$,
  and the $ x \mapsto - x $ symmetry when $\omega=\theta$, we obtain
  $ \int_{\partial B(0,R)} \langle {x}/{|x|},\omega + \theta \rangle  e^{i\lambda \langle x,\omega-\theta\rangle}ds(x)= \mathcal O ( R^{-\frac12} ) $.
  This and \eqref{e:defE} give,
  with $ u( x)   := u ( \lambda, x , \omega ) $,
  \[ \begin{split}
      & \int_{\partial B(0,R)} \left(\partial_re ( x)  e^{-i\lambda\langle x,\theta\rangle}-e ( x)  \partial_r[ e^{-i\lambda\langle x,\theta\rangle}]\right) ds(x) = \\
      & \ \ \   \int_{\partial B(0,R)}  (  \partial_r u  ( x)
      + i \lambda \langle x /|x| , \theta \rangle u ( x )  ) )e^{ - i \lambda \langle x, \theta \rangle }  ds ( x )
      + \mathcal O ( R^{-\frac12} ) .
    \end{split} \]
  In the notation of  \eqref{e:defU}, we put $ B := e^{ - \pi i /4 } \sqrt {2 \pi/ \lambda } b ( \lambda, x/|x|, \omega ) $ and then apply \eqref{eq:statm}
  to see that this is expression is equal to
  \[  e^{ i \lambda R } R^{-\frac12 } \int_{\partial B(0,R)}
    (i \lambda  + i \lambda \langle x/|x|, \theta \rangle ) B
    e^{ - i \lambda \langle x, \theta \rangle } d s ( x ) + \mathcal O ( R^{-\frac12})
    = 4 \pi i b ( \lambda , \theta , \omega ) + \mathcal O ( R^{-\frac12}).
  \]
  Combined with \eqref{e:boundaryPair} and \eqref{eq:ala} this completes the proof of
  \eqref{eq:Ala} by taking $ R \to \infty $.
\end{proof}

\noindent
{\bf Remarks.} 1. For evaluating the traces in Lemma \ref{l:tr} numerically we note that,
using a positive parametrizaton $  [ 0 , L ) \to\partial \mathscr O$, $ s \mapsto x = x ( s ) $,
$ | \dot x| = 1$, $ \nu ( s ) = ( \dot x_2 ( s ) , - \dot x_1 ( s ) ) $ ($ \nu $ is the outward normal),
\begin{equation}
  \label{eq:perp}
  \begin{gathered} \partial_\nu ( e^{  i\lambda \langle x , \omega \rangle } ) =   i \lambda \langle
    \dot x , \omega^\perp \rangle e^{  i\lambda \langle x , \omega \rangle } ,\\
    \mathbb S^1 \ni \omega =  ( \cos t , \sin t  ) , \ \ \
    \omega^\perp := ( - \sin t, \cos t ) , \ \   t \in [ 0 , 2 \pi ).
  \end{gathered}
\end{equation}

\noindent
2. We recall the following symmetry of $ e ( \lambda, x , \omega ) $ \cite[Theorem 4.20]{dizzy}:
\[  \overline{ e ( \lambda, x, \omega ) } = e ( - \lambda, x, \omega ) . \]

Next, we calculate a formula for $\sigma'(\lambda)$ in terms of $e(\lambda, x,\omega)$.
The definitions give
\begin{equation}
  \label{eq:sigp}
  \sigma'(\lambda)=\frac{1}{2\pi i} \tr S(\lambda)^*\partial_\lambda S(\lambda)= \frac{1}{2\pi i}\tr \partial_\lambda A(\lambda) +\frac{1}{2\pi i}  { \mathrm{ tr} } A(\lambda)^*\partial_\lambda A(\lambda).
\end{equation}

We start with the first term on the right hand side of \eqref{eq:sigp}:
\begin{lemm}
  \label{l:tr}
  We have
  \begin{equation}
    \label{eq:trdA}
    \tr \partial_{\lambda}A(\lambda)=\frac{1}{4\pi } \int_{\mathbb{S}^1}\int_{\partial\mathscr{O}} e^{-i\lambda \langle x,\omega\rangle} G ( \lambda, x, \omega ) ds(x)d\omega ,
  \end{equation}
  {where}, in the notation of \eqref{e:defE},
  \begin{equation}
    \label{eq:defG}
    \begin{gathered}
      G(\lambda, x, \omega):= - \langle x,\omega \rangle \partial_{\nu} u(\lambda, x, \omega) +\partial_{\nu}v(\lambda, x, \omega),  \\
      (-\Delta-\lambda^2)v(\lambda, x,\omega)=-2i\lambda u(\lambda, x,\omega),  \ \ x \in \RR^2\setminus \mathscr O, \\ v(\lambda, x, \omega)|_{\partial\mathscr O}=-\langle x,\omega \rangle e^{i\lambda \langle x,\omega \rangle}|_{\partial \mathscr O}.
    \end{gathered}
  \end{equation}
\end{lemm}
\begin{proof}
  The integral kernel of $\partial_\lambda A(\lambda)$
  is given by
  \begin{equation}
    \label{eq:kerAl}
    \partial_\lambda A(\lambda, \theta,\omega)= \frac{1}{4\pi i } \int_{\partial\mathscr{O}} \left(
    \partial_\lambda [ e^{-i\lambda \langle x,\theta\rangle} ] \partial_\nu e(\lambda,x,\omega)+ e^{-i\lambda \langle x,\theta\rangle}  \partial_\nu \partial_\lambda e(\lambda,x,\omega)] \right) ds(x).
  \end{equation}
  From \eqref{e:defE} we see that
  $ \partial_{\lambda} e( \lambda, x, \omega ) = i \langle x, \omega \rangle e^{ i \lambda \langle x, \omega \rangle} + i  v (\lambda, x, \omega) $,
  where $ v $ is defined in the statement of the lemma. Hence, in the notation of \eqref{eq:perp},
  and with
  $ e := e ( \lambda, x ,\omega ) $, the integrand in \eqref{eq:kerAl} for $ \theta = \omega $ is
  given by
  $$ % \partial_\lambda [ e^{-i\lambda \langle x,\omega \rangle} ] \partial_\nu e  +  \partial_\nu [\partial_\lambda e ] e^{ -i \lambda \langle x , \omega \rangle } = 
    i \langle \dot x, \omega^\perp \rangle + i (- \langle x , \omega \rangle \partial_\nu u ( \lambda, x , \omega ) + \partial_\nu v ( \lambda,  x, \omega ) ) e^{-i\lambda \langle x,\omega \rangle}.
  $$
  This gives \eqref{eq:trdA} since
  $ \int_{\partial \mathscr O} \langle \dot x , \omega^\perp \rangle ds = 0 $.
\end{proof}

We now move to the second term in \eqref{eq:sigp}:
\begin{lemm}
  \label{l:tr2}
  We have
  \begin{equation}
    \label{eq:trAtdA}
    \tr A(\lambda)^*\partial_\lambda A(\lambda)  = \frac{1}{16 \pi^2 }
    \int_{\mathbb S^1 } \int_{\mathbb{S}^1} H ( \lambda, \omega , \theta )
    F ( \lambda , \omega, \theta ) d\omega d \theta ,
  \end{equation}
  where in the notation of Lemma \ref{l:tr},
  \begin{equation*}
    %\label{eq:defFG}
    \begin{split}
      H & :=
      \int_{\partial \mathscr O } e^{ i \lambda \langle x, \theta \rangle}
      \left(  - i \lambda \langle \dot x, \omega^\perp \rangle e^{ - i \lambda \langle
        x , \omega \rangle } + \overline{ \partial_\nu u (\lambda, x, \omega ) }  \right) ds ( x ), \\
      F  & :=
      \int_{\partial \mathscr O } e^{ - i \lambda \langle y, \theta \rangle }
      \left[ ( \langle \dot y , \omega^\perp \rangle ( \lambda \langle y , \theta - \omega \rangle
      + i )e^{ i \lambda \langle y, \omega \rangle } - i \langle y , \theta \rangle \partial_\nu
      u ( \lambda, y, \omega ) + i \partial_\nu v ( \lambda, y , \omega ) \right] ds(y).
    \end{split}
  \end{equation*}
\end{lemm}
\begin{proof}
  The integral kernel of $A(\lambda)^*$ is given by
  $$
    A^*(\lambda,\omega,\theta)= - \frac{1}{4\pi i }\int_{\partial \mathscr O}e^{i\lambda \langle x,\theta\rangle}\partial_\nu \overline{e(\lambda,x,\omega)}ds(x),
  $$
  and hence $ \tr A(\lambda )^* \partial_\lambda A ( \lambda ) $
  is given
  as an integral over $ \partial \mathscr O_x \times \partial \mathscr O_y \times
    \mathbb S^1_{ {\theta}} \times \mathbb S^1_{{\omega}} $ of
  \[\begin{split} & \tfrac{1}{16 \pi^2 } e^{ i \lambda \langle x-y, \theta  \rangle }
      \overline{ \partial_\nu e (\lambda, x, \omega ) }
      \left( - i \langle y , \theta \rangle \partial_\nu  e( \lambda , y, \omega ) +
      \partial_\nu \partial_\lambda e ( \lambda, y, \omega ) \right) .
      % = \\
      %& \tfrac{1}{16 \pi^2 } e^{ i \lambda \langle x-y, \theta  \rangle } 
      %\left(  - i \lambda \langle \dot x, \omega^\perp \rangle e^{ - i \lambda \langle
      %x , \omega \rangle } + \overline{ \partial_\nu u (\lambda, x, \omega ) }  \right)  \times 
      %\\
      %& \ \ \ \left[ ( \langle \dot y , \omega^\perp \rangle ( \lambda \langle y , \theta - \omega \rangle 
      %+ i )e^{ i \lambda \langle y, \omega \rangle } - i \langle y , \theta \rangle \partial_\nu
      %u ( \lambda, y, \omega ) + i \partial_\nu v ( \lambda, y , \omega ) \right] = \\
      %&  \tfrac{1}{16\pi^2 } e^{ i \lambda \langle x - y , \theta \rangle }
      %\left( F ( {\lambda}, x, y, \omega ) + \langle y , \theta \rangle H (  {\lambda}, x, y , \omega ) \right), 
    \end{split} \]
  Using $ \partial_{\lambda} e( \lambda, x, \omega ) = i \langle x, \omega \rangle e^{ i \lambda \langle x, \omega \rangle} + i  v (\lambda, x, \omega)$ and the definition of $ e ( \lambda, x , \omega )$
  completes the proof.
\end{proof}

\noindent
{\bf Remark.} The integral over $ \theta $ could be eliminated using Bessel functions.
That however introduces factors
$ J_0 ( \lambda|x-y| ) $ and $\langle y , x - y \rangle J_1 ( \lambda |x-y|)/|x-y| $ and
destroys the product structure which only requires separate integration in $ x $ and $ y$.
Hence, it is not numerically advantageous.

\section{Analytic solution for the disc}

In order to validate our numerical scheme, the scheme was tested agains the analytic solution for \(\calO\) given by the unit disk. We record in this section the formulae for both $\sigma(\lambda)$ and $u(\lambda,x,\omega)$ in this case.
\subsection{The scattering phase for the unit disk}

To compute the scattering phase for the disk, we use polar coordinates and separation of variables to find the scattering matrix. In particular, in polar coordinates $(r,\theta)$, a solution to $(-\Delta-\lambda^2)u=0$ with $u|_{\partial B(0,1)}$ with $u(r,\theta)=\sum_n e^{in\theta}u_n(r)$ satisfies
$$
  \Big(-\partial_r^2-\frac{1}{r}\partial_ru+\frac{n^2}{r^2}-\lambda^2\Big)u_n(r)=0,\qquad u_n(1)=0
$$
and hence
\begin{equation}
  \label{e:separate}
  u_n(r)= A_n\Big(-\frac{H_{|n|}^{(2)}(\lambda )}{H_{|n|}^{(1)}(\lambda )}H_{|n|}^{(1)}(\lambda r)+H_{|n|}^{(2)}(\lambda r)\Big).
\end{equation}
Recall~\cite[\S 10.17(i)]{dlmf} that for $\lambda,r>0$, $n\geq 0$, we have
\begin{gather*}
  H_{n}^{(1)}(\lambda r)=\Big(\frac{2}{\pi \lambda r}\Big)^{1/2}e^{i(\lambda r-\frac{1}{2}n\pi -\frac{1}{4}\pi)}+O(r^{-3/2}),\\ H_n^{(2)}(\lambda r)=\Big(\frac{2}{\pi \lambda r}\Big)^{1/2}e^{-i(\lambda r-\frac{1}{2}n\pi -\frac{1}{4}\pi)}+O(r^{-3/2}).
\end{gather*}
Thus, $H_{|n|}^{(1)}(\lambda r)$ is outgoing and $H_{|n|}^{(2)}(\lambda r)$ is incoming and hence this implies that $\sin(n\theta)$ ($n\neq 0$) and $\cos(n\theta)$ are eigenfunctions of $S(\lambda)$ with eigenvalue
$$
  \mu_n:=(-1)^{n+1}\frac{H_{|n|}^{(2)}(\lambda )}{H_{|n|}^{(1)}(\lambda )}.
$$
In particular, using the Wronskian relation~\cite[(10.5.5)]{dlmf} in the last line, we obtain
\begin{align}
  \label{eq:analytical_dsigma}
  \sigma'(\lambda) & =\Big(\frac{1}{2\pi i}\log \det S(\lambda)\Big)'                                                                                                               \nonumber \\
                   & =\frac{R}{2\pi i}\sum_{n=-\infty}^\infty \frac{(H_{|n|}^{(2)})'(\lambda )}{H_{|n|}^{(2)}(\lambda )}-\frac{(H_{|n|}^{(1)})'(\lambda )}{H_{|n|}^{(1)}(\lambda )} \nonumber \\
                   & =-\frac{2}{\pi^2 \lambda }\sum_{n=-\infty}^\infty \frac{1}{H_{|n|}^{(1)}(\lambda )H_{|n|}^{(2)}(\lambda )}.
\end{align}

\noindent
{\bf Remark.} Note that we do not write $\sigma(\lambda)$ directly since this would involve making a choice of branch for the logarithm. We instead use the $\sigma(0)=0$ to make this choice when integrating $\sigma'(\lambda)$.

\subsection{The scattering amplitude for the unit disk}

%We use polar coordinate and separation of variables to obtain explicit expressions of \(\partial_{\nu}u\) and \(\partial_{\nu}v\). Abusing notation 
%slightly, we write  \(\omega = (\cos \omega, \sin \omega)\), \(x = ( r\cos \theta, r\sin \theta)\) and \(u(x)=f(r)g(\theta)\). Then, \eqref{eq:Helm} becomes 
%\begin{align*}
%  \left(\dfrac{\partial^2 f}{\partial r^2} + \dfrac{1}{r}\dfrac{\partial f}{\partial r}\right)g(\theta) + \dfrac{1}{r^2}f(r)\dfrac{\partial^2 g}{\partial \theta^2} + \lambda^2 f(r) g(\theta)=0, \ \ \ 
%g ( \theta + 2 \pi ) = g ( \theta ) .
%\end{align*}
%We denote \(\alpha^2\) the constant of separation, 
%\[ 
%  r^2 \dfrac{\partial^2 f}{\partial r^2}+r \dfrac{\partial f}{\partial r} + ((\lambda r)^2 - \alpha^2) f(r)=0,
%  \ \ \ 
%  \dfrac{\partial^2 g}{\partial \theta^2} + \alpha^2 g(\theta) =0.
%\]
%The outgoing solutions is given by 
%the Hankel function of the first kind, \(f(r)=H^{(1)}_{\alpha}(k r)\),  and \(g(\theta)=A_{\lvert \alpha \rvert} \cos (\lvert \alpha\rvert \theta) + B_{\lvert \alpha \rvert} \sin (\lvert \alpha\rvert \theta)\), where \(A_{\lvert \alpha \rvert}\) and \(B_{\lvert \alpha \rvert}\) are two constants. Since \(g\) is \(2\pi\)-periodic, \(\alpha\) must be an integer, and we write \(n=\alpha\). 

The the incoming portion of $e(\lambda)$ in~\eqref{e:defE} is given by the incoming portion of $e^{i\lambda \langle x,\omega\rangle}$.  Using the Jacobi–Anger expansion, with $x=r(\cos \theta,\sin \theta)$  we have
\begin{align*}
  e^{i\lambda \langle x,\omega\rangle} & = e^{i \lambda r \left(\cos \theta \cos \omega + \sin \theta \sin \omega\right)} = e^{i \lambda r \cos\left( \theta -\omega \right)} \\
                                       & = \sum_{n=0}^{\infty} \delta_n i^n \big(H^{(1)}_n(\lambda r)+H_n^{(2)}(\lambda r)\big) \cos( n(\theta -\omega) ),
\end{align*}
where \(\delta_0=\frac{1}{2}\) and \(\delta_n=1\) for \(n>0\). Thus, from~\eqref{e:separate} we have
$$
  e(\lambda,r\theta,\omega)=\sum_{n=0}^{\infty} \delta_n i^n \big(-\frac{H_n^{(2)}(\lambda)}{H_n^{(1)}(\lambda)}H^{(1)}_n(\lambda r)+H_n^{(2)}(\lambda r)\big) \cos( n(\theta -\omega) ),
$$
and hence
\begin{equation}\label{eq:analytical_solution}
  u(\lambda,r\theta,\omega)= \sum_{n=0}^\infty \delta_n i^n \Big(1-\frac{H_n^{(2)}(\lambda)}{H_n^{(1)}(\lambda)}\Big)H^{(1)}_n(\lambda r)\cos\left(n(\theta-\omega)\right).
\end{equation}
We can now easily
% Choosing \(\theta=0\), and \(\theta=\pi/(2n)\), we obtain, respectively
%\begin{align*}
%  A_n = - \delta_n (-i)^n\dfrac{J_n(kR)}{H^{(1)}_n(kR)} \cos (n\omega)\quad\text{and}\quad B_n = - \delta_n (-i)^n\dfrac{J_n(kR)}{H^{(1)}_n(kR)} \sin (n\omega).
%\end{align*}
%
%The scattering phase has the following explicit expression
%\begin{align}\label{eq:analytical_solution}
%  u ( r \cos \theta, r \sin \theta ) = - \sum_{n=0}^{\infty} \delta_n (-i)^n \dfrac{J_n(kR)H^{(1)}_n(kr)}{H^{(1)}_n(kR)} \cos(n(\theta -\omega)),
%\end{align}
%from which
deduce explicit expression for \(v\), \(\partial_{\nu}u\) and \(\partial_{\nu}v\).

\section{Numerical scheme}
\label{s:num}
In this section we describe the numerical scheme used to compute the
scattering phase.

\subsection{Setup}
To compute~\eqref{eq:trdA} and~\eqref{eq:trAtdA}, we use the trapezoidal rule to approximate the 1-d integrals along the angles \(\theta\) and \(\omega\): for \(N>0\), \(\omega_l= 2\pi l/N\) for \(l=0\cdots N-1\), and using the \(2\pi\)-periodicity, we use the following approximations
\begin{align*}
  \tr \partial_\lambda A & \approx
  \frac{1}{ 4 \pi} \dfrac{2\pi}{N} \sum_{l=0}^{N-1} \int_{\partial \mathscr{O}}
  e^{ - \lambda \langle \omega_l , x \rangle} G ( \lambda, x , \omega_l )
  d s(x),
\end{align*}
where $ G $ is given in \eqref{eq:defG}.
For the second term we benefit from the factorization in which we only compute
two integrals over the boundary:
\begin{align*}
  \tr A^* \partial_{ \lambda } A & \approx
  \frac{1}{ 16 \pi^2 }
  \left(\dfrac{2\pi}{N}\right)^2\, \sum_{l=0}^{N-1} \sum_{p=0}^{N-1}
  H ( \lambda, \omega_l, \theta_p )F ( \lambda, \omega_l, \theta_p ) ,
\end{align*}
where $ H $ and $ F $ are given in Lemma~\ref{l:tr2}. It remains compute the normal derivatives of \(u(\lambda,\cdot, \omega)\) and \(v(\lambda,\cdot, \omega)\) for \(\omega \in (\omega_l)_{l=0}^{N-1}\).

To approximate \(u\) and \(v\), we first need to reformulate both problems on a bounded domain in \(\mathbb{R}^2\setminus\overline{\mathscr{O}}\). We use the method of \emph{Perfectly Matched Layers} (PML) (introduced in~\cite{Ber} for electromagnetic waves) to do this. More precisely, we use a radial PML~\cite{CoMo}: consider a disk $B_{R_{\mathrm{PML}}}$ with \(R_{\mathrm{PML}}>R_{\mathrm{DOM}}\) such that $\overline{\mathscr{O}}\subsetneq B_{R_{\mathrm{DOM}}}$, we reformulate both~\eqref{eq:Helm} and~\eqref{eq:defG}  using polar coordinates $(r,\theta)$ in $B_{R_{\mathrm{PML}}}$, and we apply a complex scaling $\hat{r}=r +\frac{i}{\lambda}\int_0^r \gamma(s)ds$ where \(\gamma\) is an increasing function defined on $[0,R_{\mathrm{PML}})$ and equal to zero in $[0,R_{\mathrm{DOM}})$. Several choices can be made for $\gamma$, we choose \(\gamma(r):=1/(R_{\mathrm{PML}}-r)\) for \(r\in [R_{\mathrm{DOM}},R_{\mathrm{PML}})\) as advocated in~\cite{Ber+}. We denote \(\mathbf{J}_{\mathrm{PML}}\) the Jacobian of the transformation from the Cartesian coordinates to the complexified Cartesian coordinates.

The equations for $ u $ and $v$, \eqref{eq:Helm} and~\eqref{eq:defG} are solved with the Galerkin method using Lagrange finite elements; i.e.\ we solve these equations in a finite-dimensional subspace \(V_h\subset H^1(B_{R_{\mathrm{PML}}}\setminus \overline{\mathscr{O}})\) formed by piecewise-polynomial functions on a mesh, and we denote \(h\) the mesh element size (see~\cite{ErGu} for more information): we find \(u_h, v_h \in V_h\) such that \(u_h|_{\partial\mathscr{O}}= -\mathcal{I}_h (e^{i\lambda \langle x,\omega\rangle})|_{\partial \mathscr{O}}\), \(v_h|_{\partial\mathscr{O}}= -\mathcal{I}_h(\lambda \langle x,\omega\rangle e^{i\lambda \langle x,\omega\rangle})|_{\partial \mathscr{O}}\) where \(\mathcal{I}_h:C^0(\overline{B_{R_{\mathrm{PML}}}\setminus \overline{\mathscr{O}}})\to V_h\) is the Lagrange interpolation operator, \(u_h|_{\partial B_{\mathrm{PML}}}=v_h|_{\partial B_{\mathrm{PML}}}=0\),
\begin{align*}
  a(u_h,w_h)=0\text{ for all }w_h \in V_{h,0}, \text{ and } a(v_h,w_h)=b_{u_h} (w_h)\text{ for all } w_h \in V_{h,0},
\end{align*}
where \(V_{h,0}\) is the subspace of functions in \(V_h\) whose value on \(\partial \mathscr{O} \cup \partial B_{\mathrm{PML}}\) is zero,
\begin{align*}
  a(u,w)     & = \int_{B_{R_{\mathrm{DOM}}} \setminus \overline{\mathscr{O}}} (\nabla u \cdot \nabla w - \lambda^2 u w)dxdy                                                                                                                     \\
             & + \int_{B_{R_{\mathrm{PML}}} \setminus \overline{B_{R_{\mathrm{DOM}}}}} (\mathbf{J}_{\mathrm{PML}}^{-T}\nabla u \cdot \mathbf{J}_{\mathrm{PML}}^{-T}\nabla w - \lambda^2 u w) \lvert \det  \mathbf{J}_{\mathrm{PML}}\rvert dxdy, \\
  b_{u_h}(w) & = - 2 i \lambda \int_{B_{R_{\mathrm{PML}}}} u_h w \lvert \det  \mathbf{J}_{\mathrm{PML}}\rvert dxdy.
\end{align*}

In our numerical experiments, the approximation space \(V_h\) is spanned by  $\mathbb{P}_2$ Lagrange elements, i.e.\ continuous piecewise quadratic functions. To bound the error from discretization independently of \(\lambda\) when solving~\eqref{eq:Helm} and~\eqref{eq:defG}, we need \(h^{2p}\lambda^{2p+1}=h^{4}\lambda^{5}\) bounded~\cite{DuWu}, where \(h\) is the mesh size and \(p\) is the degree of the finite element functions. To satisfy this condition, we set the number of points per wavelength to \(\mu \times (1+\lambda^{1/4})\), where \(\mu\) is a constant. Differentiating \(u_h\) and \(v_h\) to take the Neumann trace on \(\partial \mathscr{O}\), we obtain $\mathbb{P}_1$ Lagrange elements on the discretization of \(\partial \mathscr{O}\), which can then be used to compute \(G ( \lambda, x , \omega_l )\), \(H ( \lambda, \omega_l, \theta_p )\) and \(F ( \lambda, \omega_l, \theta_p )\).

Note that these approximations depend on \(\lambda\) and the angle \(w_l\) in the Dirichlet conditions, and thus require solving~\eqref{eq:Helm} and~\eqref{eq:defG} for \(N\) different angles and hence $N$ different right-hand sides, for a given frequency \(\lambda\). Thus, for a given \(\lambda\), we factorize the matrix stemming from the discretization (note that it is the same for both \(u_h\) and \(v_h\)), and we use it to solve the discretized problems with several right-hand sides at the same time to improve efficiency. The numerical computations were carried out with FreeFEM~\cite{hec}. More precisely, we used its interface with PETSc~\cite{bal} to solve linear systems with MUMPS~\cite{Am1,Am2}.

\noindent
{\bf Remark.}
Since we only need the Neumann traces of $u$ and $v$ to compute the scattering phase, it is quite natural to want to reformulate both problems~\eqref{eq:Helm} and~\eqref{eq:defG} using Boundary Integral Equations (BIE). While~\eqref{eq:Helm} can easily be reformulated with a standard BIE, the presence of a right-hand side in~\eqref{eq:defG} makes it less convenient to usual boundary integral formulations. Nevertheless, it should be possible to represent \(v\) differentiating Green's third identity (which we can use to represent \(u\)), but it would imply non-standard boundary integral operators. Thus, we preferred to use more standard tools such as PML.

\subsection{Convergence}

When \(\mathscr{O}\) is a disk, we use the analytical expression from \eqref{eq:analytical_dsigma}, with a truncated sum using \(\lvert n\rvert\leq 5\lambda\), to compute the relative error on \(\sigma'\). In Table~\ref{table:cvg_disk_lambda_mu}, from left to right, the frequency \(\lambda\) is increasing. The tables at the top have \(R_{\mathrm{PML}}-R_{\mathrm{DOM}}=0.25\), while tables at the bottom keep a number of mesh cells in the PML region constant, \(R_{\mathrm{PML}}-R_{\mathrm{DOM}}=5h\).

For a fixed \(R_{\mathrm{PML}}-R_{\mathrm{DOM}}\) and \(\lambda\) increasing (tables at the top in Table~\ref{table:cvg_disk_lambda_mu}), the error is decreasing, which is consistent with~\cite{gal}, which states that the error on \(u\) should decrease in this case. We also observed that keeping a fixed number of mesh cells in the PML region (tables at the bottom in Table~\ref{table:cvg_disk_lambda_mu}) is enough to have the same level of precision as with a fixed PML region. This is due to the particular choice of \(\gamma\), and we do not observe this behaviour with other usual complex scaling (taking \(\gamma\) as a linear or quadratic function for example). The advantage is that, in this case, \(R_{\mathrm{PML}}-R_{\mathrm{DOM}}\) decreases so that the computational cost is reduced compared to keeping \(R_{\mathrm{PML}}-R_{\mathrm{DOM}}\) constant.

Table~\ref{table:cvg_disk_N} gives the relative error on \(\sigma'\) with \(N\) increasing, \(\mu=20\), \(R_{\mathrm{DOM}}=2\) and \(R_{\mathrm{PML}}-R_{\mathrm{DOM}}=5h\). We observe that we need to take \(N\) large enough to converge to the same level of error as in Table~\ref{table:cvg_disk_lambda_mu}, and \(N\) needs to be larger for larger \(\lambda\): \(N=30\) for \(\lambda=10\) and \(N=50\) for \(\lambda=10\). This is consistent with the fact that \(u\) and \(v\) are more and more oscillatory when \(\lambda\) increases, and we observed numerically that taking \(N\sim \lambda\) is sufficient to keep the error bounded independently of \(\lambda\).

\begin{figure}
  \includegraphics[width=14cm]{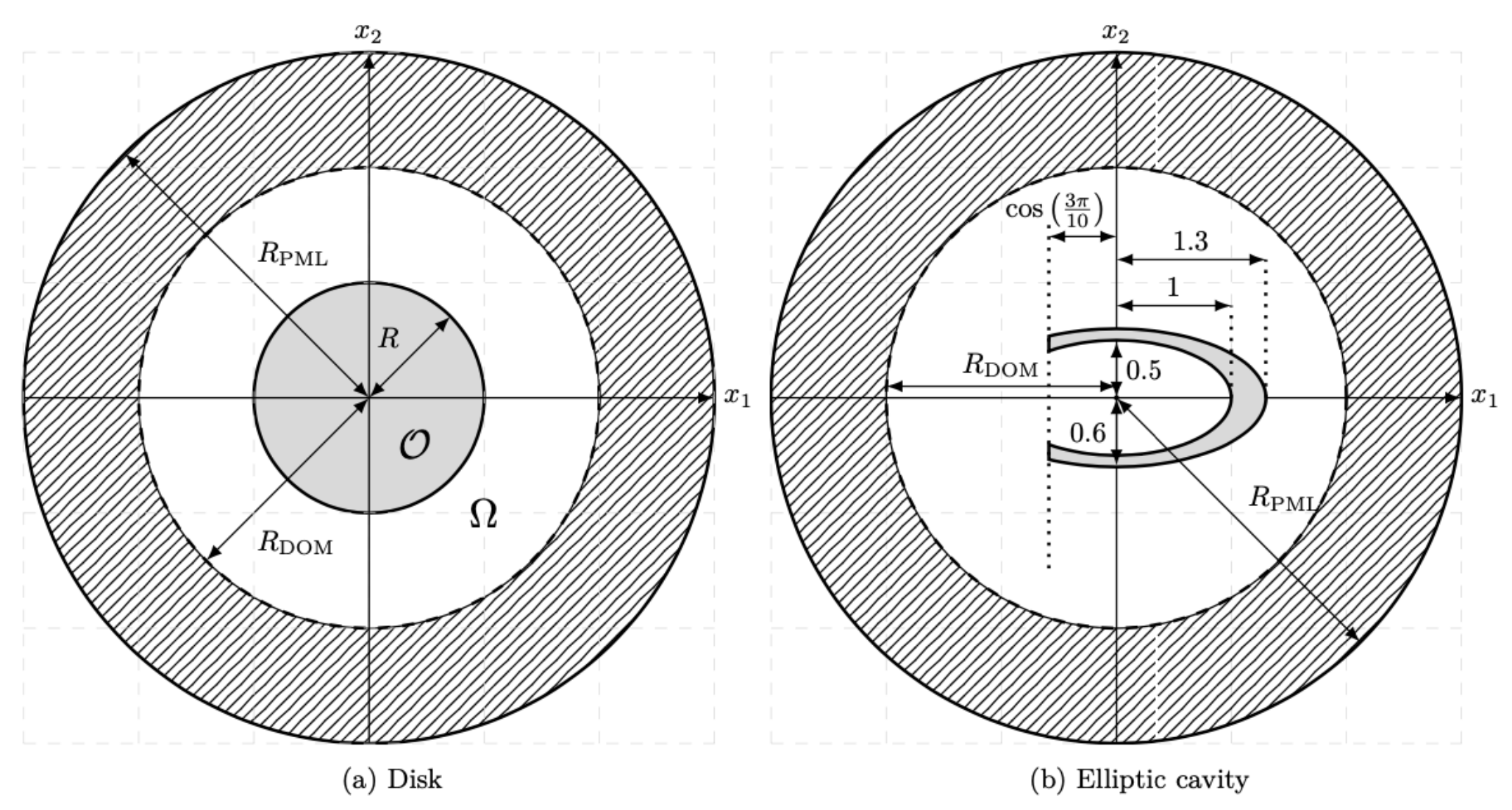}
  \caption{Considered geometries with their PML}
\end{figure}
\begin{table}
  \tabcolsep=0.10cm
  \begin{subtable}[T]{0.49\textwidth}
    \centering
    \begin{tabular}{cccc}
      \hline
      $\mu$ & Relative error on \(\sigma'\) \\
      \hline
      1     & 0.1519                        \\
      5     & 0.0120                        \\
      10    & 0.0038                        \\
      15    & 0.0023                        \\
      20    & 0.0015                        \\
      \hline
    \end{tabular}
    \caption*{\(\lambda=10\), \(R_{\mathrm{PML}}-R_{\mathrm{DOM}}=0.25\)}
  \end{subtable}
  \begin{subtable}[T]{0.49\textwidth}
    \centering
    \begin{tabular}{cccc}
      \hline
      $\mu$ & Relative error on \(\sigma'\) \\
      \hline
      1     & 0.0258                        \\
      5     & 0.0097                        \\
      10    & 0.0030                        \\
      15    & 0.0016                        \\
      20    & 0.0008                        \\
      \hline
    \end{tabular}
    \caption*{\(\lambda=20\), \(R_{\mathrm{PML}}-R_{\mathrm{DOM}}=0.25\)}
  \end{subtable}\par\bigskip
  \begin{subtable}[T]{0.49\textwidth}
    \centering
    \begin{tabular}{cccc}
      \hline
      $\mu$ & Relative error on \(\sigma'\) \\
      \hline
      1     & 0.0779                        \\
      5     & 0.0108                        \\
      10    & 0.0038                        \\
      15    & 0.0021                        \\
      20    & 0.0015                        \\
      \hline
    \end{tabular}
    \caption*{\(\lambda=10\), \(R_{\mathrm{PML}}-R_{\mathrm{DOM}}=5h\)}
  \end{subtable}
  \begin{subtable}[T]{0.49\textwidth}
    \centering
    \begin{tabular}{cccc}
      $\mu$ & Relative error on \(\sigma'\) \\
      \hline
      1     & 0.0334                        \\
      5     & 0.0096                        \\
      10    & 0.0030                        \\
      15    & 0.0015                        \\
      20    & 0.0008                        \\
      \hline
    \end{tabular}
    \caption*{\(\lambda=20\), \(R_{\mathrm{PML}}-R_{\mathrm{DOM}}=5h\)}
  \end{subtable}
  \caption{Relative error on \(\sigma'\) for a disk with \(R_{\mathrm{DOM}}=2\) and \(N=100\).}
  \label{table:cvg_disk_lambda_mu}
\end{table}

\begin{table}
  \begin{subtable}[T]{0.49\textwidth}
    \centering
    \begin{tabular}{cccc}
      \hline
      $\mu$ & \(N\) & Relative error on  \(\sigma'\) \\
      \hline
      20    & 20    & 0.0594                         \\
      20    & 25    & 0.0025                         \\
      20    & 30    & 0.0015                         \\
      20    & 35    & 0.0015                         \\
      20    & 40    & 0.0015                         \\
      20    & 45    & 0.0015                         \\
      20    & 50    & 0.0015                         \\
      20    & 55    & 0.0015                         \\
      20    & 60    & 0.0015                         \\
      \hline
    \end{tabular}
    \caption*{\(\lambda=10\)}
  \end{subtable}
  \begin{subtable}[T]{0.49\textwidth}
    \centering
    \begin{tabular}{cccc}
      \hline
      $\mu$ & \(N\) & Relative error on  \(\sigma'\) \\
      \hline
      20    & 20    & 0.0618                         \\
      20    & 25    & 0.0310                         \\
      20    & 30    & 0.0309                         \\
      20    & 35    & 0.0311                         \\
      20    & 40    & 0.0307                         \\
      20    & 45    & 0.0031                         \\
      20    & 50    & 0.0008                         \\
      20    & 55    & 0.0008                         \\
      20    & 60    & 0.0008                         \\
      \hline
    \end{tabular}
    \caption*{\(\lambda=20\)}
  \end{subtable}
  \caption{Relative error on \(\sigma'\) for a disk with \(R_{\mathrm{DOM}}=2\) and \(R_{\mathrm{PML}}-R_{\mathrm{DOM}}=5h\)}
  \label{table:cvg_disk_N}
\end{table}

\subsection{Main numerical results}

The values of \(\sigma'\) in Figure~\ref{f:nontrap} are obtained for \(\lambda\geq 3\) with \(\mu=30\), \(R_{\mathrm{PML}}-R_{\mathrm{DOM}}=5h\) and \(N= 10\lambda\). For \(0.3\leq \lambda<3\), we computed \(\sigma'\), but this required the use of significantly larger \(\mu\): usually \(\mu=300\) for \(0.3 \leq \lambda \leq 2\) and \(\mu=200\) for \(2 \leq \lambda \leq 3\). Figure~\ref{f:trap} was produced in the same way, except that we took \(\mu=100\) away from an interval of size 0.2 centered on the quasimode frequencies (which are explicitly computeable using the eigenvalues of the Laplacian in the ellipse, see~\cite[Section 1.1.3]{MGSS}). On the intervals near quasimode frequences we also needed to increase \(\mu\) significantly, and we took \(\mu=300\). For every geometry, we refined the mesh around corners in order to obtain good precision.

\end{document}